\documentclass[10pt]{amsart}

\usepackage{amsmath, amssymb, amscd}
\usepackage{eucal}
\usepackage{mathrsfs}
\usepackage[frame,cmtip,arrow,matrix,line,graph,curve]{xy}
\usepackage{graphicx}

\usepackage[colorlinks=true,citecolor=red,linkcolor=blue]{hyper ref}

\newcommand{\cL}{\mathcal L}

\newcommand{\cN}{\mathcal N}

\newcommand{\cO}{\mathcal O}

\newcommand{\hgamma}{\hat{\gamma}}
\newcommand{\hcL}{\hat{\cL}}
\newcommand{\hM}{\hat{M}}

\newcommand{\tiota}{{\tilde \iota}}
\newcommand{\tM}{\widetilde M}

\newcommand{\tcL}{{\widetilde {\cL}}}

\newcommand{\qq}{\mathbb Q}

\newcommand{\cc}{\mathbb C}
\newcommand{\rr}{\mathbb R}
\newcommand{\zz}{\mathbb Z}

\newcommand{\ep}{\epsilon}

\newcommand{\inv}{^{-1}}

\newtheorem{prop}{Proposition}[section]
\newtheorem{thm}[prop]{Theorem}
\newtheorem{lem}[prop]{Lemma}
\newtheorem{cor}[prop]{Corollary}

\theoremstyle{remark}
  \newtheorem{rk}[prop]{Remark}
  
\theoremstyle{definition}
 
 \newtheorem{defn}[prop]{Definition}

\numberwithin{equation}{section}

\begin{document}
\title[Dynamics of riemannian $1$-foliations on $3$-manifolds]
{Dynamics of riemannian $1$-foliations on $3$-manifolds}
\author{Jaeyoo Choy$^{\dagger}$ and Hahng-Yun Chu$^{\ast}$}
\address{Department of Mathematics, Kyungpook National University, 80 Daehak-ro, Buk-gu,
Daegu 41566, Republic of Korea}\email{jaeyoochoy@gmail.com\ (J. Choy)}
\address{Department of Mathematics, Chungnam National University, 79 Daehak-ro, Yuseong-Gu,
Daejeon 34134, Republic of Korea}\email{hychu@cnu.ac.kr\ (H.-Y. Chu)}

 

\thanks{$^\dagger$ The first author}
\thanks{$^\ast$ The corresponding author}

\begin{abstract}
In this paper we study several dynamical properties of a riemannian $1$-dimensional foliation $\cL$ on an oriented closed $3$-manifold $M$.
Carri\`ere \cite{Ca} classified such pairs $(M,\cL)$.
Using the classification we describe in detail recurrence points, $\omega$-limit sets and attractors.
Finally,  using the fact that the Poincar\'e map on a tranversal surface for a riemannian $1$-dimensional foliation is an isometry, we show the nonhyperbolicity of $(M,\cL)$.

\end{abstract}

\maketitle

\thispagestyle{empty} \markboth{Jaeyoo Choy and Hahng-Yun Chu} {Dynamics of riemannian $1$-foliations on $3$-manifolds}

\section{Introduction}\label{sec: intro}

\medskip

This paper aims to understand dynamics of riemannian $1$-dimensional foliations ($1$-foliations for short) on $3$-manifolds.
We focus on the notions of hyperbolicity, recurrence and limit sets, especially $\omega$-limit sets and attractors, of the foliations on the phase spaces.
The study of the above concepts is quite classic and important in the field of dynamical systems.

Let $\cL$ be a $1$-foliation on an oriented closed connected $3$-manifold $M$.
We say that $\mathcal{L}$ is $\textit{riemannian}$ if there is a riemannian metric on the normal bundle in the sense of Carri\`ere (see Definition \ref{def: riemann} for the precise definition).
In \cite{Ca}, Carri\`{e}re classifies all the oriented closed $3$-manifolds equipped with a riemannian $1$-foliations.
Using this classfication we prove several dynamical properties of $(M,\cL)$.

Regarding continuous dynamical systems, we mainly deal with the study of behavior of $1$-foliations in a smooth manifold. On the systems, the notion of hyperbolicity is a special key to understand the figure of the given $1$-foliation. Note that a $1$-foliation is called \textit{Anosov $1$-foliation} on a manifold $M$ if the manifold $M$ is uniformly hyperbolic for the $1$-foliation \cite{An}. Under the assumption of uniform hyperbolicity, there are abundant information about the $1$-foliation to describe the behavior. See \cite{Bo} for details. Therefore we pursue direction towards dynamics equivalent to the concept of hyperbolicity on some region of a given manifold. For the purpose we approach through the concept of partially hyperbolic $1$-foliations in generic dynamics. In recent years, there appear plenty of results about this notion. See \cite{AMS}\cite{SSW}.

To be specific, recurrence and $\omega$-limit set are central properties in dynamical systems.
In analysis the behavior of the orbits of $1$-foliations, they play important roles.
In the theorems listed in the below, we first see that for a riemannian $1$-foliation on an oriented closed $3$-manifold $M$, the whole manifold $M$ turns out to be recurrence. And then we classify the $\omega$-limit sets of the phase space $(M,\cL)$.

\medskip

\begin{thm}\label{th: Rec th}

Assume that $\mathcal{L}$ is a
riemannian $1$-foliation on an oriented closed $3$-manifold $M$.
Then every point of $M$ is a recurrence point.
\end{thm}

\medskip

To study dynamical systems, another important aspect which is dealt with in this paper is attractor. In this paper, we follow Conley's definition of attractors. For the detail, see \cite{Co}.
In \cite{MP}, it is proved that in compact $3$-manifolds, the nontrival attractor is mixing for a generic $3$-dimensional $1$-foliations. In \cite{MPP}, the authors describe the robust transitivity for $3$-manifolds. In our paper we will see that a robust transitive set containing singularities of a $1$-foliation on a closed $3$-manifold is either a proper attractor or a proper repeller. Hence we conclude the nonexistence of proper attractor for a riemannian $1$-foliation on closed $3$-manifolds as follows.

\medskip

\begin{thm}\label{th: limit}

Let $(M,\cL)$ be as in Theorem \ref{th: Rec th}.
Then the $\omega$-limit sets are diffeomorphic to \textit{either} a circle $S^1$ \textit{or} a $2$-torus $T^2$ or $M$ itself.
And there does not exist any proper nonempty attractor.
\end{thm}

\medskip

In fact the above submanifolds $S^{1}, T^{2}$ in Theorem \ref{th: limit} will become clearer in \S\ref{subsec: Carriere} because these depend on the examples in Carri\`ere's list.

Our study on the riemannian $1$-foliations on $3$-manifolds is partly motivated by the classification theory of Brunella \cite{Br1} and Ghys \cite{Gh}.
In their works $\cL$ is assumed to be transversely holomorphic (see Definition \ref{def: trans holo} for the precise definition).
This notion was studied by Haefliger-Sundararaman \cite[\S1]{HS} in more general context.
In \cite{Br1}\cite{Gh}, Brunella and Ghys classify
all the transversely holomorphic $1$-foliations on closed $3$-manifolds in detail.
Specifically when $H^2(M,\cO)\neq0$, Ghys shows that all the transversely holomorphic $1$-foliations are riemannian.
Here $\cO$ denotes the sheaf of functions on $M$ constant along the leaves of $\cL$ and homomorphic in a transverse direction (see Definition \ref{def: cO}).
So our results are also true for any transversely holomorphic $1$-foliation with $H^{2}(M,\cO)\neq0$ (Corollary \ref{cor: summary}).

Finally the following theorem of the paper asserts the nonhyperbolicity of the foliations on $3$-dimensional manifolds.

\medskip

\begin{thm}\label{th: nonhyperbolic}
Let $(M,\cL)$ be as in Theorem \ref{th: Rec th}.
Then $(M,\cL)$ is nonhyperbolic with respect to any riemannian metric on $M$.
\end{thm}

\medskip

In the case $H^{2}(M,\cO)=0$, Brunella classifies all the pairs $(\hM,\hcL)$ where $\hM$ and $\hcL$ are obtained by `complexification of leaves' of Haefliger-Sundararaman \cite[Proposition 2.1]{HS}.
See \S\ref{sec: Brunella} for the further details.
However at this moment we do not understand the dynamical properties as in the above case $H^{2}(M,\cO)\neq0$.
\vskip.2cm

\textit{Contents of this paper.}
In \S\ref{sec: def and notation} we define riemannian $1$-foliation, transversely holomorphic $1$-foliation and above dynamical properties.
In \S\ref{sec: proof} we introduce Carri\`ere's classification and then prove the main theorems on the dynamical properties.
In \S\ref{sec: Brunella} we introduce some parts of Brunella and Ghys' works \cite{Br1}\cite{Gh}.
This section itself does not contain a new result but is required only for the reformulation of the main theorems in terms of transversely holomorphic foliations in our context.
Thus we introduce a relevant part of their results for completeness' sake although we do not treat Brunella's case $H^{2}(M,\cO)=0$.


\section{Definitions, notations and basic properties}\label{sec: def and notation}


In this section we set up basic definitions and notations and then give related properties.


\subsection{Definitions and notations for dynamical properties}

Let $\cL:\rr\times M\to M$ be a $1$-foliation on Hausdorff topological space $M$.
We denote by $\cL^t:=\cL(t,\bullet):M\to M$ for short.
We define the $\textit{$\omega$-limit set}$ and
$\textit{$\alpha$-limit set}$ of $q $ by
$$ \omega(q):= \{ x \in M : x=\lim_{n\rightarrow\infty} \cL^{t_n}(q)
\mbox{ \ for some sequence}\ t_{n}\rightarrow \infty \mbox{ as}\
n\rightarrow\infty \},$$
$$ \alpha(q):= \{ x \in M : x=\lim_{n\rightarrow\infty} \cL^{-t_n}(q)
\mbox{ \ for some sequence}\ t_{n}\rightarrow \infty \mbox{ as}\
n\rightarrow \infty \}.$$

\medskip

\begin{defn}
A point $x \in M$ is
\textit{$\omega$-recurrent} or \textit{positively recurrent} with
respect to $\cL^{t}$ if $x \in \omega (x)$ and is
\textit{$\alpha$-recurrent} or \textit{negatively recurrent} with
respect to $\cL^{t}$ if $x \in \alpha (x)$. A point $x \in M$ is
\textit{(Poincar\'e) recurrent} with respect to $\cL^{t}$ if $x$
is simultaneously positively and negatively recurrent with respect to
$\cL^{t}$.
\end{defn}

\medskip

\begin{defn}\label{def: hyp diffeo} Let $\cL:\rr\times M\to M$ be a $1$-foliation on a finite dimensional smooth manifold $M$. A compact $\cL$-invariant set, $\Lambda \subset M$, is called a \textit{hyperbolic set} for the $1$-foliation $\cL$
if there exist $C > 0$ and $0 < \lambda < 1$ such that for each elemnet $x$ of $\Lambda$, there exist a decomposition
$$
T_x M = E_x^{ss} \oplus E_x^{uu}  \oplus E_x^c
$$
such that $\partial_t \cL(t,x) | _{t=0} \in E_x^c-\{ 0 \}$, $\dim (E^c(x))=1$, $D_x \cL(t, x)(E_x^i) = E_x^i$ with $i= ss, uu$, and
$$
\parallel D_x \cL(t, x)|E_x^{ss} \parallel \leq C \lambda^t \textit{ \ for \ } t \geq 0,
$$
$$
\parallel D_x \cL(t, x)|E_x^{uu} \parallel \leq C \lambda^t \textit{ \ for \ } t \leq 0
$$
where $\parallel \cdot \parallel$ is a norm induced by the
riemannian metric.
If the whole manifold $M$ is a hyperbolic set for $\cL$, then the $\mathcal{C}^1$ $1$-foliation $\cL$ is called an $\textit{Anosov 1-foliation}$.

\medskip

Let $\cL$ be an Anosov $1$-foliation on a compact connected manifold $M$. The bundle $E^{ss}$ and $E^{uu}$ are called $\textit{strong stable bundle}$ and $\textit{strong unstable bundle}$ of $\cL$, respectively.
So one may say that a derivative $D_x \cL$ of a $1$-foliation $\cL$ is
\textit{eventually contracting} on $E_x^{ss}$ and \textit{eventually
expanding} on $E_x^{uu}$.

\end{defn}

On a closed manifold $M$, an Anosov $1$-foliation $\cL^t$ means that it is uniformly hyperbolic
on the whole manifold $M$.
Note that the hyperbolicity for diffeomorphisms is similarly defined by replacing $\cL^t$
into $f$ for $0<t\ll1$.
See the definition \ref{def: hyp diffeo}.
The above choice of $t$ depends on $v\in M$.

As aforementioned, another important part in the study of dynamical system is to understand the structures of the attractors and $\omega$-limit sets of the dynamical system (see e.g.\ \cite{Mi}\cite{CM}).
We basically follow Conley's definition of attractors in \cite{Co} but slightly modified.
Note that even before Conley's definition many other
definitions of attractor can be found in several papers (cf.\ \cite{Mi}).
See also \cite{Choy} for relations among them.
Now we define an attractor.

\smallskip

\begin{defn}\label{def: attractor}
Let $(X,d)$ be a metric space, and $\phi^t$ be a continuous $1$-foliation
on $X$. A nonempty open subset $U$ of $X$ is an \textit{attractor
block} for $\phi^t$ if $\overline{\phi^t(U)}\subseteq U$ for every
$t>0$. A proper subset $A$ of $X$ is called an \textit{attractor} for $\phi^t$
if there exists an attractor block $U$ satisfying
$$A=\bigcap_{t\geq 0}\overline{\phi^{t}(U)}.$$
\end{defn}


\subsection{Definitions and basic properties for riemannian $1$-foliations and transversely holomorphic $1$-foliations}
\label{subsec: riemannian $1$-foliation}

In this subsection we define and study riemannian $1$-foliations, transversely holomorphic $1$-foliations and related properties.
We reemphasize that the readers can skip `transversely holomorphic' if they are not interested in the reformulations of the main theorems in terms of transversely holomorphic foliations.

Let $M$ be a smooth manifold.
Let
$\cL$ be a nondegenerate $1$-foliation on $M$, i.e.\
a nowhere vanishing smooth section of the tangent bundle $TM$.
The corresponding $1$-foliation $\cL^t$ has nowhere vanishing derivative with respect to $t$.

We set up notation for Poincar\'e (first-return) map.
For simplicity we assume $M$ is of dimension $3$ from now on.
Our argument also works for any dimension.
Fix a point $x\in M$.
Let $D_x$ be any embedded unit disk in $M$ centered at $x$.
We further assume any tangents of $D_x$ are transverse to $\cL$.
Let $t_0>0$ such that $\cL^{t_0}(x)\in D_x$ but $\cL^t(x)\notin D_x$ for any $t\in(0,t_0)$.
Then there is an open neighborhood $U$ of $x$ in $D_x$ such that the map assigning $y\in U$ to the first touching point $\cL^t(y)$ in $D$ $(t>0)$ is a diffeomorphism from $U$ onto the image.
We denote this map by 
	$$
	\rho_{D_x}: U\to D_x
	$$ 
(called the \textit{Poincar\'e map}).

To define riemannian $1$-foliations, first we need to introduce a pseudo-group.
Let $D_0$ be the unit disk in $\rr^2$ and $\iota_{x}:D_0\to D_x$ be the diffeomorphism for the above embedded disk $D_x$ for each $x\in M$.
Let $\ep_x>0$ such that $\iota_x$ extends to an (unique) open embedding
    $$
    \tiota_x:D_0\times(-\ep_x,\ep_x)\to M\ \mbox{with}\  \tiota_x(y,t)=\cL^t(\iota_x(y))\  (t\in(-\ep_x,\ep_x))
    $$
for all $y\in D_0$.
I.e., the $(-\ep_x,\ep_x)$-direction is parallel to the $1$-foliation of $\cL$.
We denote by $U_x$ the $\tiota_x$-image open subset of $M$.
It is clear that $\{U_x\}_{x\in M}$ is an open cover of $M$.
Let $p_x:U_x\to D_0$ be the composite of $\tiota_x|_{U_x}^{-1}$ with the projection to $D_0$.
Let $\gamma_{xy}$ be the composite $p_x\circ\iota_y$.
We need a care here: this composition is \textit{not} well-defined over $D_0$ but only over the (possibly empty) subset $\iota_y^{-1}(U_x)$.

We define the composition $\gamma_{zx}\circ\gamma_{xy}\, (= p_z\circ \iota_x\circ p_x\circ \iota_y)$ by shrinking domains appropriately.

\medskip

\begin{lem}\label{lem: composition}
By shrinking domains appropriately as above we have
$p_z=p_z\circ \iota_x\circ p_x$.
\end{lem}

\smallskip

\begin{proof}
We notice that under the identification via $\tiota_x$, $p_x$ maps any point $x'$ to the unique intersection point $x''$ in $D_x$ with the leaf containing $x'$.
Similarly
$p_z$ maps $x''$ to another intersection point in $D_z$ with the leaf containing $x''$.
It is clear that both leaves are same as they contain $x''$.
Therefore we obtain the equivalence of the maps in the lemma.
\end{proof}

\medskip

By this lemma the composite $\gamma_{zx}\circ\gamma_{xy}$ coincides with $\gamma_{zy}$ over an open subset in $D_0$.
Consequently we have the group law in the set
    $$
    \Gamma:=\{\gamma_{xy}\}_{x,y\in M}
    $$
with the obvious identity and inverses.
This is called a (holonomy) pseudo-group.
The name `holonomy' pseudo-group will become clearer due to Lemma \ref{lem: riemann}.

\medskip

\begin{defn}\label{def: riemann}
We say $\cL$ is \textit{riemannian} if there exist embeddings $\iota_x: (D_0,0)\to (M,x)$ $(x\in M)$ and a riemannian metric on $D_0$ invariant under all the locally defined diffeomorphisms in the associated pseudo-group $\Gamma$.
\end{defn}

\medskip

We construct a fibre-wise metric of the normal bundle $\cN$ of $\cL$ when $\cL$ is riemannian.
First note that $D_0$ forms a disk bundle on $M$ using the transitions $\gamma_{xy}$.
Since $\gamma_{xy}$ preserves the metric $g$ on $D_0$, the disk bundle has the induced fibre-wise metric.
Since $D_x$ transverses the $1$-foliation direction of $\cL$ for each $x\in M$, the disk bundle is naturally isomorphic to $\cN$.

\medskip

\begin{lem}\label{lem: riemann}
If $\cL$ is riemannian, there is the induced smoothly varying fibre-wise metric on $\cN$.
\end{lem}

\begin{proof}
The unproven part is the smoothness.
But this is clear by giving the same metric $g$ on ${t}\times D_0$ for each $t\in(-\ep_x,\ep_x)$.
\end{proof}

\medskip

We define transversely holomorphic $1$-foliations following \cite[\S1]{HS}.

\medskip

\begin{defn}\label{def: trans holo}
We identify $\rr^2$ with $\cc$ and thus $D_0$ is an analytic open subset.
$\mathcal{L}$ is \textit{transversely holomorphic} if it is nondegenerate and there exist embeddings $\iota_x: (D_0,0)\to (M,x)$ $(x\in M)$ satisfying $\gamma_{xy}$ of the pseudo-group $\Gamma$ are all holomorphic maps.
\end{defn}

\medskip

\begin{rk}
Due to the orientations of $\cc$ and $\rr$, any transversely holomorphic $1$-foliation induces orientation of $M$.
\end{rk}

\medskip


\section{Dynamical properties of transversely holomorphic $1$-foliations \label{sec: appl.}}\label{sec: proof}

We prove the various dynamical properties of transversely holomorphic $1$-foliations stated in Introduction (\S1).

\subsection{Proof of Theorems \ref{th: Rec th} and \ref{th: limit} via the classification of Carri\`ere}
\label{subsec: Carriere}

Now we prove Theorems \ref{th: Rec th} and \ref{th: limit} using Carri\`ere's classification of the pairs $(M,\cL)$ of a closed $3$-manifold and a riemannian flow.
We need a lemma:

\medskip

\begin{lem}\label{lem: quotient}
Let $(\tM,\tcL)$ be a pair of a smooth manifold and a $1$-foliation.
Suppose that a discrete group $G$ acts freely on $\tM$.
Let $M:=\tM/G$ the group quotient and $\phi\colon \tM\to M$ be the quotient map.
Assume also that there is a $1$-foliation $\cL$ on $M$ such that $\phi$ maps the flow $\tcL^t$ to $\cL^t$.
Then we have the following:

$(1)$
The $\phi$-image of a recurrence point of $(\tM,\tcL)$ (resp.\ an $\omega$-limit set and an attractor) is also a recurrence point $(M,\cL)$ (resp.\ an $\omega$-limit set and an attractor).
And any attractor of $(M,\cL)$ is given as the $\phi$-image of some attractor of $(\tM,\tcL)$.

$(2)$
Moreover if $G$ is a finite group, an $\omega$-limit set of $(M,\cL)$ is the $\phi$-image of some $\omega$-limit set of $(\tM,\tcL)$.
\end{lem}

\proof
The former item (1) is clear.
We prove the latter item (2).
Let $A$ be an attractor of $(M,\cL)$.
Since $\phi$ is a covering map, $\phi\inv(A)$ is also an attractor.
We consider an $\omega$-limit set $\omega(x)$ where $x\in M$.
We choose any $y\in \phi\inv(x)$.
We claim that the $\phi$-image of $\omega(y)$ coincides with $\omega(x)$.
First we observe that $\phi(\omega(y))\subset \omega(x)$.
Indeed for any point $z\in \omega(y)$ there exists a sequence $t_n$ with $z=\lim_{n\to\infty}\tcL^{t_n}(y)$.
Thus $\phi(z)=\lim_{n\to\infty}\cL^{t_n}(x)$.
We prove the opposite inclusion $\phi(\omega(y))\supset \omega(x)$.
Let $z\in \omega(x)$.
Then there exists a sequence $t_n$ with $z=\lim_{n\to\infty}\cL^{t_n}(x)$.
By the unique lifting property of covering space of a path, $\cL^t(x)$ lifts to $\tcL^t(y)$ for any given $y\in\phi\inv(x)$.
Since a $\phi$-fibre is a finite set, there exists a subsequence $t_{n_k}$ with a limit $\lim_{n\to\infty}\tcL^{t_{n_{k}}}(x)$.
Since $\phi$ is a local isomorphism, it is clear that the $\phi$-image of the limit coincides with $z$.
\qed\vskip.3cm

Any pair $(M,\cL)$ of a closed $3$-manifold and a riemannian flow is one of the following list \cite{Ca}:

\begin{enumerate}
\item
$M$ is a $3$-torus $T^3$
and $\cL$ is linear with an irrational slope on $T^3$ (i.e.\ the corresponding flow gives translation of $T^3$).
\item
$M$ is also $T^3$.
But there are precisely two possibilities of $\cL$ as follows:
\begin{itemize}
\item[(a)]
Let $\cL'$ be a linear $1$-foliation on a $2$-torus $T^2$ with an irrational slope.
Let $M:=T^2\times S^1$.
Then $\cL$ is the foliation such that for each $x\in S^1$, the induced flow $\cL^t$ of $\cL$ lies in $T^2\times\{x\}$ and it coincides with $(\cL')^t$.
\item[(b)]
We fix any $a\in SL_2(\zz)$ with eigenvalues $\lambda,\lambda\inv$ where $\lambda>1$.
We denote the corresponding eigenvectors by $v,v'\in\rr^2$ respectively.
We consider the linear foliation $\cL'$ on $\rr^2$ whose time-$1$ map maps $0$ to $v$.
Since $\lambda$ is known be an irrational number, $\cL'$ has an irrational slope.
We consider a $\zz^2$-action and a $\zz$-action on $\rr^3=\rr^2\times \rr^1$ as follows:
The $\zz^2$-action comes from the standard affine translation group action on the first factor $\rr^2$ and the $\zz^1$-action is defined by $n.(m,t)=(a^n(m),t+n)$.
The $(\zz^2\ltimes\zz)$-quotient of $\rr^3$ becomes a $3$-torus.
Let $M$ be this $3$-torus.
Let us construct a $1$-foliation $\cL$ on $M$ as follows:
Let $\cL''$ be the foliation such that for each $x\in \rr^1$, its flow $(\cL'')^t$ lies in $\rr^2\times\{x\}$ and it coincides with $(\cL')^t$.
Now $\cL$ is set to the induced $1$-foliation from $\cL''$ on the discrete group quotient $M$.
\end{itemize}
\item
there are two cases:
\begin{itemize}\item[(a)] $M$ is a lenz space $L_{p,q}$ ($p,q\in \zz\setminus0$).
It is defined as the quotient of the $3$-sphere
    $$
    S^3=\{(z_1,z_2)\in \cc^2|\ |z_1|^2+|z_2|^2=1\}
    $$
by a $\zz^1$- action
    $$
    n.(z_1,z_2)= \left(e^\frac{2n\pi i}{p}z_1,e^\frac{2n\pi i}{q}z_2\right).
    $$
Let $\lambda,\mu\in\rr\setminus\qq$ with $\lambda/\mu\in\rr\setminus \qq$.
Let $\cL'_{(\lambda,\mu)}$ be the foliation on $S^3$ whose corresponding flow is given by
    $$
    \cL^t(z_1,z_2)=(e^{i\lambda t} z_1,e^{i\mu t}z_2),\quad t\in\rr.
    $$
Note that $\lambda/\mu$ is the slope of $\cL'_{(\lambda,\mu)}$.
Let $\cL$ be the induced foliation on the discrete group quotient $M$ from $\cL'_{(\lambda,\mu)}$.
\item[(b)]
Let $M=S^2\times S^1$.
Fixing the north and south poles of $S^2$ we consider any flow on $S^2$ given by rotation.
Let us consider the corresponding foliation $\cL'$ on $S^2$.
Let $\cL$ be the foliation such that for $x\in S^1$, its flow $\cL^t$ lies in $S^2\times\{x\}$ and it coincides with $(\cL')^t$.
\end{itemize}
\item
$M$ is a Seifert fibration, i.e.\
an $S^1$-fibration over a smooth $2$-manifold.
And $\cL$ is a foliation such that its flow $\cL^t$ lies in the fibre direction.
\end{enumerate}

We proceed the proof of Theorems \ref{th: Rec th} and \ref{th: limit} in each case.

In the case (1) all the positive orbits of the corresponding flow $\cL^t$ to $\cL$ are dense.
Thus every point of $M$ is a recurrence point.
This completes the proof of Theorem \ref{th: Rec th} in this case.
By the same reason the $\omega$-limit set of any point coincides with $M$ itself.
Similarly no nonempty proper subset of $M$ can be an attractor.
This completes the proof of Theorem \ref{th: limit} in this case.

We prove the case (2)(b).
We observe that $M$ is a (nontrivial) $T^2$-fibration over $S^1$.
To be precise the projection $\rr^2\times \rr^1\to\rr^1$ induces $\phi\colon M=(\rr^2\times\rr^1)/(\zz^2\ltimes\zz) \to \rr^1/\zz$ where the $\zz$-action on $\rr^1$ is the standard translation action.
The fibres of the induced map are $T^2=\rr^2/\zz^2$.
By construction of $\cL$ every positive orbit closure becomes the $T^2$-fibre itself containing the orbit.
Thus every point of $M$ is a recurrence point.
By a similar argument in the case (1) we also deduce that the $\omega$-limit set of any point is the $T^2$-fibre containing the point.

Now we prove that there are no nonempty proper attractors in the case (2)(b).
Let $A$ be a nonempty proper attractor if any.
Let $U$ be an attractor block of $A$.
We take an open subset $V$ in the base $S^1$ such that $\phi|_{\phi\inv(V)}$ gives a trivial $T^2$-fibration $\phi\inv(V)$ intersects $A$ but is not contained in $A$.
We denote by $A',U'$ the intersections of $A,U$ with $\phi\inv(V)\cap A,\phi\inv(V)\cap U$ respectively.
Let $p\in A'$.
Since any positive orbit lies in a $T^2$-fibre, $A'$ contains the $\omega$-limit $\omega(p)$.
Thus $A'$ should be of the form $T^2\times B$ for some proper relatively closed subset $B$ in $V$.
Then there exists a proper open subset $V'$ in $V$ with $B\subset V'$.
Thus we have
    $$
    A'=\bigcap_{t\ge0}\overline{\cL^t(U)}
    =T^2\times\left(\bigcap_{t\ge0}\overline{\cL^t(V')}\right) =T^2\times V'.
    $$
Here we used $\cL^t(V')=V'$ for any $t\ge0$ since the flow is parallel to the $T^2$-fibres.
This is contradiction because $T^2\times V'$ properly contains $A'$.

The case (2)(a) is easier because the $T^2$-fibration is trivial.
Thus a similar argument also works as in the case (2)(b)
This completes the proof of Theorems \ref{th: Rec th} and \ref{th: limit} in the case (2).

Let us prove the case (3)(a).
Assume first $p=q=0$ so that $M=S^3$ (although this is not a subcase of (3)(a)).
Let
    $$
    T_k=\left\{(z_1,z_2)\in S^3\middle|
    |z_1|^2=k,\ |z_2|^2=1-k\right\}
    $$
where $0\le k\le 1$.
Note that the flow $\cL^t(x)$ for any $x\in T_k$ lies in $T_k$ for each $k$.
If $k=0,1$ then $T_k=S^1$ and it coincides with the positive orbit.
Thus every point of $T_k$ is a recurrence point and its $\omega$-limit set is $T_k$ itself.
If $0<k<1$ then $T_k=T^2$ and any positive orbit in $T_k$ is dense in $T_k$ because $\cL^t$ has an irrational slope $\lambda/\mu$ by the assumption.
Thus every point of $T_k$ is also a recurrence point and its $\omega$-limit set is $T_k$ itself.
Let us show there is no nonempty proper attractor.
Suppose not, so there exists a nonempty proper attractor $A$.
Let $U$ be an attractor block of $A$.
Note that the map $\phi:M=S^3\to [0,1],\ (z_1,z_2)\mapsto |z_1|^2$ restricts to $T^2$-fibration over the smaller base $(0,1)$.
Note that $T_k=\phi\inv(k)$.
We denote by $A', U'$ the intersection of $A,U$ with this $T^2$-fibration respectively.
Then $U'$ is also an attractor block of $A'$ with respect to the restricted foliation.
Note that $U'$ is nonempty since $\dim U'=3$ but $M\setminus \phi\inv(0,1)=T_0\sqcup T_1$ has dimension $1$
However $A'$ is possibly empty.
By a similar argument in the product space as in the above proof of (2)(b), $\bigcap_{t\ge0}\overline{\cL^t(U')}$ properly contains $A'$ because $\phi$ restricts to the trivial fibration in our case.
This contradiction completes the proof of Theorem \ref{th: limit} in this case.
Now we get back to the original assumption $(p,q)\neq(0,0)$.
Then $M$ is a finite group quotient of $S^3$ and thus the same result as above holds after passing to the quotient due to Lemma \ref{lem: quotient}.

In the case (3)(b) it is clear that every point is a recurrence point since the flow is induced from rotation of the $S^2$-factor.
The description of $\omega$-limit sets is also clear:
If $p\in M=S^2\times S^1$ projects to the north or south pole of $S^2$, $\omega(p)$ is $S^1$.
Otherwise it is $T^2$.
There is no nonempty proper attractor by a similar argument in the above proof of the case $p=q=0$.
We omit the details.

In the case (4), since  $\cL^{t}(x)$ for any $x\in M$ lies in the $S^1$-fibre containing $x$, the positive orbit of each $p\in M$ is the $S^1$-fibre containing $p$.
Thus $p$ is a recurrence point and $\omega(p)$ is the $S^1$-fibre.
There is no nonempty proper attractor by a similar argument in the product space in the above.
This completes the proof of Theorems \ref{th: Rec th} and \ref{th: limit}.

\subsection{Nonhyperbolicity and proof of Theorem \ref{th: nonhyperbolic}}

We first define hyperbolicity for a diffeomorphism from $M$ to itself.

\medskip

\begin{defn} Let $(M,f)$ be a pair of a finite
dimensional smooth manifold and a
diffeomorphism on $M$.
We say that $M$ has a \textit{hyperbolic
structure} with respect to $f$ if there is a riemannian metric on $M$ and
a continuous splitting of $TM$ into the direct sum of
$Tf$-invariant subbundles $E^{s}$ and $E^{u}$ such that for some
constants $A$ and $0<\lambda<1$ and for all $v \in E^{s}$, $w \in
E^{u}$ and $n\geq 0$,
\begin{equation}\label{eq21}
\parallel Tf^{n}(v)\parallel \leq A \lambda^{n}  \parallel v \parallel, \ \
\parallel Tf^{-n}(w)\parallel \leq A \lambda^{n}  \parallel w \parallel,
\end{equation}
where $\parallel \cdot \parallel$ is a norm induced by the
riemannian metric. Thus one may say that $Tf$ is
\textit{eventually contracting} on $E^{s}$ and \textit{eventually
expanding} on $E^{u}$.

A \textit{hyperbolic subset} of $M$ with respect to $f$ is a
closed invariant subset of $M$ with hyperbolic structure with respect to the restriction of $f$.
\end{defn}

\medskip

To prove Theorem \ref{th: nonhyperbolic} we need a lemma:

\medskip

\begin{lem}\label{lem: holo}
Suppose $M$ is a closed $3$-manifold and $\cL$ is riemannian.
The Poincar\'e map $\rho_{D_x}$ on the embedded disk $D_x=\iota_x(D_0)$ for any $x\in M$ is an isometry where $\iota_x$ is given as in Definition \ref{def: riemann}.

Hence it is nonhyperbolic.
\end{lem}

\begin{proof}
First we notice that the $\rho_{D_x}$ is not vacuous for some $x$.
This comes from the generality that the compactness of $M$ implies the existence of recurrent point (ref.\ \cite[pp.101]{AS}).
In fact every point of $M$ is a recurrent point as we will see in \S\ref{subsec: Carriere}.
Thus for any point $x\in M$, $\rho_{D_x}$ is defined.

By Lemma \ref{lem: riemann} the induced metric $D_x$ does not depend on $x$.
This implies $\rho_{D_{x}}$ is an isometry.

The nonhyperbolicity is clear from isometry.
\end{proof}

\medskip

Let $\cO$ be a sheaf of germs of functions on $M$
which are constant along the leaves and holomorphic in the
transverse direction.

\medskip

\begin{thm}[Ghys' theorem \cite{Gh}]
Let $\mathcal{L}$ be a transversely holomorphic foliation on a closed $3$-manifold $M$. If $H^{2}(M,\cO)\neq 0$, $\mathcal{L}$ is riemannian. \vskip.3cm
\end{thm}

Combining Lemma \ref{lem: holo} and Ghys' theorem we have an immediate corollary:

\medskip

\begin{cor} \label{thm: nonhyperbolicitysheaf}
Let $\mathcal{L}$ be a
transversely holomorphic foliation on a closed
$3$-manifold $M$. If $H^{2}(M, \cO)\neq 0$, the Poincar\'e map $\rho_{D_x}$
for each recurrent point $x$ is nonhyperbolic with respect to the induced riemannian metric.

Hence $\rho_{D_x}$ is nonhyperbolic with respect to any riemannian metric.
\end{cor}

\medskip

Now we prove Theorem \ref{th: nonhyperbolic}.
\vskip.3cm

\textit{Proof of Theorem \ref{th: nonhyperbolic}.}
We notice that a tubular neighborhood of the path defining $\rho_{D_x}$ has the induced riemannian metric on $M$ as it is diffeomorphic to $\cN$ along the path.
From this metric we can construct a metric on $M$ such that its restriction to a smaller tubular neighborhood of the path still has the same induced metric.
This is obtained by using the standard argument of partition of unity, so we omit the detail.

By Lemma \ref{lem: holo}, $\rho_{D_x}$ is nonhyperbolic and thus so is
$(M,\cL)$ with respect to the metric.
In general the inequality \eqref{eq21} is independent of the choice of metric.
So $(M,\cL)$ is also nonhyperbolic with respect to any riemannian metric.
This completes the proof of the theorem.
\qed


\section{Transversely holomorphic $1$-foliations}\label{sec: Brunella}

In this section we consider transversely holomorphic $1$-foliations instead of riemannian foliations and then deduce Corollary \ref{cor: summary} in the below from Theorems \ref{th: nonhyperbolic}--\ref{th: limit}.
The proof itself was already given in \S\ref{sec: intro} (after the statement of Theorem \ref{th: limit}) using the known fact \cite{Gh}: $H^2(M,\cO)\neq0$ implies $\cL$ is riemannian.
So it remains to explain necessary notions around the sheaf $\cO$.
For the purpose we need to introduce relevant parts in Brunella and Ghys' works \cite{Br1}\cite{Gh}.


\subsection{Complexification of leaves and the sheaf $\cO$}

We fix $(M,\cL)$ where $\cL$ is a transversely holomorphic $1$-foliation.
Ghys and Brunella used the harmonic transition data of $(M,\cL)$ whose local version appeared as the complexification of leaves of Haefliger-Sundararaman  \cite[Proposition 2.1]{HS}.

We use the notation in \S\ref{subsec: riemannian 1-foliation}.
We will take soon an open embeddings $\tiota_{x}\colon D_{0}\times (-\ep_{x},\ep_{x})\to M$ where $x\in M$ so that they form \textit{harmonic atlas} of $M$ in the sense of Brunella \cite{Br1}.
But for any choice of $\tiota_{x}$ we denote the image of $\tiota_{x}$ and the open subset $p_{x}(U_{x}\cap U_{y})$ of $D_{0}$ by $U_{x}$ and $D_{xy}$ respectively.

\medskip

\begin{prop}\label{prop: harmonic}\cite[\S1]{Br1}
If $\cL$ is a transversely holomorphic $1$-foliation, there exist $\tiota_{x}$ $(x\in M)$ such that the functions $\gamma_{xy}$ and $\tilde h_{xy}$ are all holomorphic and harmonic respectively where
	$$
	\begin{aligned}
	\tiota_x\inv\circ\tiota_y
	\colon
	&
	D_{xy}\times (-\epsilon_{y},\epsilon_{y})\to D_{xy}\times (-\epsilon_{x},\epsilon_{x}),
	\\
	&
	(z,t)\mapsto (\gamma_{xy}(z),t+\tilde{h}_{xy}(z)).
	\end{aligned}
    	$$
\end{prop}

\medskip

\begin{defn}\label{def: cO}
$\cO$ denotes the sheaf on $M$ such that for each open subset $U$ of $M$, $\cO(U)$ is the set of functions on $U$ such that $f\circ \tiota_{x}|_{\tiota_{x}\inv(U_{x}\cap U)}$ is  holomorphic on $D_{0}$ and constant along $(-\ep_{x},\ep_{x})$ (the flow direction) for each $x\in M$.
\end{defn}

\medskip

By compactness of $M$ we take finitely many $x_{i}$ such that $U_{x_{i}}$ forms an open cover of $M$.
Accordingly we change the index $x$ so that $\gamma_{ij}$ denotes $\gamma_{x_{i}x_{j}}$, etc.
Let $\ep:=\min_{i}(\ep_{x_{i}})$.
Let $\tilde k_{ij}$ be a harmonic conjugate of $\tilde{h}_{ij}$ so that $\tilde h_{ij}+\sqrt{-1}\tilde k_{ij}\colon (-\ep,\ep)\times S^{1}\to\cc$ is holomorphic where $ (-\ep,\ep)\times S^{1}$ is identified as an annulus in $\cc$.
Thus we have extension of $\tiota_{i}\inv\circ\tiota_{j}$ as
	$$
	\begin{aligned}
	\hgamma_{ij} \colon
	&
	D_{ij}\times(-\epsilon,\epsilon)\times S^1\to D_{ij}\times(-\epsilon,\epsilon)\times S^1,
	\\
	&
	(z,t,s)\mapsto (\gamma_{ij}(z),t+\tilde{h}_{ij}(z),s+\tilde{k}_{ij}(z))
	\end{aligned}
	$$
(see \cite[p.276]{Br1}).


\subsection{Results of Ghys and Brunella}

In general $\hgamma_{ij}$ do \textit{not} satisfy the $1$-cocycle condition:
    $$
    \hat{\gamma}_{ik}=\hat{\gamma}_{ij}\circ\hat{\gamma}_{jk}.
    $$
By \cite[Proposition 3]{Br1} if $H^{2}(M,\cO)=0$, they form a $1$-cocycle.
In other words the compact complex surface $\hM$ given by the transitions $\hgamma_{ij}$ is naturally an $S^{1}$-fibration over $M$.
In addition there is a complexification $\hcL$ of $\cL$, i.e., a holomorphic $1$-foliation on $\hM$.
In the complex surface theory there are fundamental classification results of Kodaira--Enrique and Inoue \cite{In} for the compact complex surfaces with a non-degenerate holomorphic $1$-foliation.
Thus the list of $(M,\cL)$ in \cite[Theorem 1]{Br1} is obtained.

On the other hand if $H^2(M,\cO)\neq0$, $\cL$ is riemannian by \cite[Theorem 1.1]{Gh}.
This is how we deduced in \S\ref{sec: intro} the following:

\medskip

\begin{cor}\label{cor: summary}
Let $(M,\cL)$ be a pair of an oriented closed $3$-manifold and a transversely holomorphic $1$-foliation on $M$.
If $H^{2}(M,\cO)\neq0$, the assertions of Theorems \ref{th: limit} and \ref{th: nonhyperbolic} also hold.
\end{cor}


\end{document}